\documentclass[12pt,a4paper]{amsart}

\usepackage[abbrev]{amsrefs}
\usepackage{amscd}

\theoremstyle{plain}
  \newtheorem{thm}{Theorem}[section]
  \newtheorem{lem}[thm]{Lemma}
  
  \newtheorem{cor}[thm]{Corollary}
  \newtheorem{prop}[thm]{Proposition}

\theoremstyle{definition}
  \newtheorem{defn}[thm]{Definition}
  
  \newtheorem{ex}[thm]{Example}

\theoremstyle{remark}
  \newtheorem{rem}[thm]{Remark}

  \newtheorem*{ack}{Acknowledgment}

\numberwithin{equation}{section}

\DeclareMathOperator{\diam}{diam}
\DeclareMathOperator{\supp}{supp}

\DeclareMathOperator{\ObsDiam}{ObsDiam}
\DeclareMathOperator{\me}{me}
\DeclareMathOperator{\Sep}{Sep}

\DeclareMathOperator{\dconc}{{\it d}_{{\rm conc}}}
\newcommand{\Lip}{\mathcal{L}{\it ip}}
\newcommand{\CP}{\C P}
\newcommand{\RP}{\R P}
\newcommand{\HP}{\H P}

\newcommand{\cL}{\mathcal{L}}
\newcommand{\cM}{\mathcal{M}}
\newcommand{\cP}{\mathcal{P}}

\newcommand{\cX}{\mathcal{X}}

\newcommand{\field}[1]{\mathbb{#1}}

\newcommand{\C}{\field{C}}
\newcommand{\R}{\field{R}}

\renewcommand{\H}{\field{H}}

\begin{document}

\title[limit formulas]
{Limit formulas for metric measure invariants
and phase transition property}

\thanks{The authors are partially supported by a Grant-in-Aid
for Scientific Research from the Japan Society for the Promotion of Science.
The first author is supported by Research Fellowships of the Japan Society for the Promotion of Science for Young Scientists.}

\begin{abstract}
  We generalize the observable diameter and the separation distance
  for metric measure spaces to those for pyramids, and prove
  some limit formulas for these invariants 
  for a convergent sequence of pyramids.
  We obtain various applications of our limit formulas as follows.
  We have a criterion of the phase transition property
  for a sequence of metric measure spaces or pyramids,
  and find some examples of symmetric spaces of noncompact type
  with the phase transition property.
  We also give a simple proof of a theorem
  in \cite{FnSy} on the limit of an $N$-L\'evy family.
\end{abstract}

\author{Ryunosuke Ozawa \and Takashi Shioya}

\address{Mathematical Institute, Tohoku University, Sendai 980-8578,
  JAPAN}

\date{\today}




\maketitle

\section{Introduction}
\label{sec:intro}

The study of Gromov-Hausdorff convergence of Riemannian manifolds
is one of the central topics in Riemannian geometry.
For a Gromov-Hausdorff convergence,
the upper bound of dimension is necessary for various reasons.
One of the main reasons is that the Gromov-Hausdorff precompactness
cannot be expected for a sequence of manifolds with unbounded dimension.
Different from the Gromov-Hausdorff metric,
Gromov \cite{Gmv:greenbook}*{\S 3.$\frac12$} introduced the observable distance function,
say $\dconc$, on the set, say $\cX$, of mm-spaces (metric measure spaces),
based on the idea of the concentration of measure phenomenon due to
L\'evy and Milman.  He constructed a natural compactification, say $\Pi$,
of $(\cX,\dconc)$,
which is useful to describe the asymptotic behavior of a sequence of
Riemannian manifolds with unbounded dimension.
In general, the limit of a sequence of manifolds
is no longer an mm-space and is an element of $\Pi$.
We have $\{S^n(\sqrt{n})\}_{n=1}^\infty$ as a typical example of such a sequence,
where $S^n(r)$ denotes an $n$-dimensional sphere of radius $r$
in the Euclidean space $\R^{n+1}$.
The sequence $\{S^n(\sqrt{n})\}$ converges to an element of $\Pi$,
called the \emph{virtual infinite-dimensional standard Gaussian space},
which is the infinite-dimensional version of a Euclidean space
with the standard Gaussian measure (see \cites{Sy:mmlim,Sy:book}).
On the other hand, $\{S^n(\sqrt{n})\}$ is not Gromov-Hausdorff precompact
and has no Gromov-Hausdorff convergent subsequence. 

The observable diameter and the separation distance
are two of the most important and fundamental invariants of an mm-space.
It is a natural problem to investigate the limit of these two invariants
for a convergent sequence of mm-spaces.
In this paper, we generalize these two invariants to those for an element of $\Pi$,
and prove some formulas for the limit of these two invariants
for a convergent sequence in $\Pi$.
We apply these formulas to study the asymptotic behavior of a sequence
of Riemannian manifolds with unbounded dimension.
The L\'evy family property and
the $\infty$-dissipation property for a sequence of mm-spaces (pyramids)
are two of the extremal properties in the asymptotic behavior.
The L\'evy family property corresponds to condensation
and the $\infty$-dissipation property does to evaporation.
We consider a property like the phase transition for a sequence of
mm-spaces or pyramids, say the \emph{phase transition property}.
We obtain a useful criterion for the phase transition property,
and prove that some symmetric spaces of compact type
have the phase transition property.

We describe more details for the compactification $\Pi$ of $\cX$.
For two mm-spaces $X$ and $Y$,  we define that $X \prec Y$ holds if
there is a $1$-Lipschitz map from $Y$ to $X$ that pushes the measure on $Y$
forward to that on $X$.
This is a partial order relation, called the \emph{Lipschitz order relation}.
We define a \emph{pyramid} to be a family of mm-spaces
forming a directed set with respect to the Lipschitz order
and with some closedness condition (see Definition \ref{defn:pyramid}).
For example, for a given mm-space $X$, the family
\[
\cP_X := \{\;Y \in \cX \mid Y \prec X\;\}
\]
is a pyramid, say the \emph{pyramid associated with $X$}.
The compactification $\Pi$ of $\cX$ is, in fact, realized as the family of pyramids.
It has a natural metric and the map
\[
\iota : \cX \ni X \longmapsto \cP_X \in \Pi
\]
is a topological embedding map.
The pyramid associated with a one-point mm-space $*$
is $\cP_* = \{*\}$, which is the minimal pyramid with respect to
the inclusion relation.
The family $\cX$ itself is the maximal pyramid.
A sequence of mm-spaces (resp.~pyramids) is a \emph{L\'evy family}
if and only if it converges to a one-point mm-space
(resp.~$\cP_*$).
See Corollary \ref{cor:Levy}.
A sequence of mm-spaces (resp.~pyramids) \emph{$\infty$-dissipates}
if and only if the sequence of pyramids associated with them
(resp.~the sequence itself) converges to the maximal pyramid $\cX$
(see Lemma \ref{lem:dissipate}).

We generalize the observable diameter and the separation distance
to those for a pyramid, and prove the following limit formulas.
Denote by $\ObsDiam(\cP;-\kappa)$ and
$\Sep(\cP;\kappa_0,\kappa_1,\dots,\kappa_N)$
the observable diameter and the separation distance of a pyramid $\cP$,
respectively (see Definitions \ref{defn:ObsDiam}, \ref{defn:Sep},
\ref{defn:ObsDiam-pyramid}, and \ref{defn:Sep-pyramid}).
Convergence in $\Pi$ is called \emph{weak convergence}.

\begin{thm}[Limit formulas]
  \label{thm:lim}
  Let $\cP$ and $\cP_n$, $n=1,2,\dots$, be pyramids.
  If $\cP_n$ converges weakly to $\cP$ as $n\to\infty$, then
  \begin{align*}
    \ObsDiam(\cP;-\kappa)
    &= \lim_{\varepsilon\to 0+} \liminf_{n\to\infty}
    \ObsDiam(\cP_n;-(\kappa+\varepsilon)) \\
    &= \lim_{\varepsilon\to 0+} \limsup_{n\to\infty}
    \ObsDiam(\cP_n;-(\kappa+\varepsilon)),\\
    \Sep(\cP;\kappa_0,\kappa_1,\dots,\kappa_N)
    &= \lim_{\varepsilon\to 0+} \liminf_{n\to\infty}
    \Sep(\cP_n;\kappa_0-\varepsilon,\kappa_1-\varepsilon,\dots,\kappa_N-\varepsilon) \\
    &= \lim_{\varepsilon\to 0+} \limsup_{n\to\infty}
    \Sep(\cP_n;\kappa_0-\varepsilon,\kappa_1-\varepsilon,\dots,\kappa_N-\varepsilon)
  \end{align*}
  for any $\kappa,\kappa_0,\dots,\kappa_N > 0$.
\end{thm}

Elek \cite{Ek} proved a similar result, which is only an inequality
and for another compactification of the space $\cX$ with a stronger topology.
He also assumes the boundedness of diameter for a sequence of mm-spaces,
so that $\{S^n(\sqrt{n})\}$ cannot be treated in his result.

For the proof of Theorem \ref{thm:lim}, we introduce a new metric
on $\Pi$ using measurements,
and prove some formulas between the metric and
the observable diameter/the separation distance.

We consider the limit behavior of a given sequence
of mm-spaces (or pyramids) under scale changes.
For $t > 0$ and a pyramid $\cP$,
we denote by $t\cP$ the scale change of $\cP$ with factor $t$.
We define that a sequence of pyramids $\cP_n$, $n=1,2,\dots$,
has \emph{the phase transition property}
if there is a sequence of positive real numbers $c_n$, $n=1,2,\dots$, such that
\begin{enumerate}
\item if $t_n/c_n \to 0$ as $n\to\infty$, then
  $\{t_n\cP_n\}$ is a L\'evy family;
\item if $t_n/c_n \to +\infty$ as $n\to\infty$, then
  $\{t_n\cP_n\}$ $\infty$-dissipates.
\end{enumerate}
We call such a sequence $\{c_n\}$ a sequence of \emph{critical scale order}.
The second named author proved in \cites{Sy:mmlim,Sy:book}
that the sequences of spheres $S^n(1)$ and complex projective spaces $\CP^n$
both have the phase transition property with critical scale order $\sim \sqrt{n}$.
Note that there are many examples of manifolds
that do not have the phase transition property.
We intuitively expect spaces with high symmetry
to admit the phase transition property.
We apply the limit formulas (Theorem \ref{thm:lim})
to obtain the following criterion for the phase transition property.

\begin{thm}[Criterion for phase transition property] \label{thm:phase}
  Let $\{\cP_n\}$ be a sequence of pyramids.
  Then the following {\rm(1)} and {\rm(2)} are equivalent to each other. 
  \begin{enumerate}
  \item $\{\cP_n\}$ has the phase transition property.
  \item There exists a sequence $\{r_n\}_{n=1}^\infty$ of positive real numbers
  such that
  \[
  \ObsDiam(\cP_n;-\kappa) \sim r_n
  \]
  for any $\kappa$ with $0 < \kappa < 1$,
  where $a_n \sim b_n$ means that the ratios $a_n/b_n$
  and $b_n/a_n$ are bounded.
  \end{enumerate}
  In this case, $\{1/r_n\}$ is a sequence of critical scale order.
\end{thm}

Note that (2) of the theorem means that the order of $\ObsDiam(X_n;-\kappa)$
as $n\to\infty$ is independent of $\kappa$.
The theorem is a first discovery for the value of the
lower estimate of the observable diameter.


We give an application of Theorem \ref{thm:phase}.
Let $\RP^n$, $\CP^n$, $\HP^n$ denote
the $n$-dimensional real, complex, and quaternionic projective spaces,
respectively.
$SO(n)$, $SU(n)$, and $Sp(n)$ denote
the special orthogonal group of order $n$,
the special unitary group of order $n$,
and the compact symplectic group of order $n$.
$V_k(\R^n)$, $V_k(\C^n)$, and $V_k(\H^n)$ denote
the real, complex, and quaternionic Stiefel manifolds, respectively.
We equip them with the Riemannian distance function
and the normalized Riemannian volume measure.

\begin{cor} \label{cor:sym-sp}
  Let $\{k_n\}$ be a sequence of natural numbers with $k_n \le n$.
  The sequences $\{S^n(1)\}$, $\{\RP^n\}$, $\{\CP^n\}$, $\{\HP^n\}$,
  $\{SO(n)\}$, $\{SU(n)\}$, $\{Sp(n)\}$,
  $\{V_{k_n}(\R^n)\}$, $\{V_{k_n}(\C^n)\}$, and $\{V_{k_n}(\H^n)\}$
  all have the phase transition property
  of critical scale order $\sim \sqrt{n}$.
\end{cor}

The corollary for $\{S^n(1)\}$ and $\{\CP^n\}$ is already known
as in \cite{Sy:mmlim}*{Theorem 1.1}.

%
%
%

The limit formulas are also useful to study
an $N$-L\'evy family,
which is defined in \cite{FnSy} by using the separation distance
(see Definition \ref{defn:N-Levy}).
A $1$-L\'evy family coincides with a L\'evy family.
A typical example of an $N$-L\'evy family
is a sequence of closed Riemannian manifolds $M_n$, $n=1,2,\dots$, such that
the $N$-th nonzero eigenvalue of the Laplacian on $M_n$
is divergent as $n\to\infty$.
We prove that the limit of an $N$-L\'evy family is the pyramid
associated with some finite extended mm-space
consists of at most $N$ points (see Corollary \ref{thm:N-Levy}),
where `extended' means that
the distance between two points is allowed to be infinity.
By using this statement, we give a simple proof of
\cite{FnSy}*{Theorem 4.4} (see Corollary \ref{cor:shrink}).

As another application of the limit formulas,
we are able to estimate the observable diameter of the $l_p$-product $X^n$
of an mm-space $X$, which together with Theorem \ref{thm:phase}
leads us to the phase transition property of the sequence of
the $l_p$-product $X^n$, $n=1,2,\dots$.
We here assume the discreteness of $X$ in the case of $p > 1$
for the lower estimate of the observable diameter.
This study is published separately as \cite{OzSy:prod}.

This paper is organized as follows.
In \S\ref{sec:prelim}, we describe basic definitions and facts
in metric measure geometry.
In \S\ref{sec:ObsDiam}, we define the observable diameter of a pyramid and
prove the limit formula for observable diameter.
In \S\ref{sec:Sep}, we define the separation distance
of a pyramid and prove the limit formula for separation distance.
In \S\ref{sec:N-Levy}, we study an $N$-L\'evy family
and prove that the limit of an $N$-L\'evy family is realized
by a finite extended mm-space.
In particular, we see that a L\'evy family of pyramids converges
to a one-point mm-space.
In \S\ref{sec:phase},
we study dissipation and prove Theorem \ref{thm:phase}.
Applying Theorem \ref{thm:phase},
we give several examples with the phase transition property.

\begin{ack}
  The authors would like to thank Prof.~Takefumi Kondo
  and Dr.~Yu Kitabeppu for inspiring discussions.
\end{ack}

\section{Preliminaries} \label{sec:prelim}

In this section, we give the definitions and the facts
stated in \cite{Gmv:greenbook}*{\S 3$\frac12$}.
In \cite{Gmv:greenbook}*{\S 3$\frac12$}, many details are omitted.
We refer to \cite{Sy:book} for the details.
The reader is expected to be familiar with basic measure theory
and metric geometry (cf.~\cites{Kch, Bil, Bog, BBI}).

\subsection{mm-Isomorphism and Lipschitz order}

\begin{defn}[mm-Space]
  An \emph{mm-space} is defined to be a triple $(X,d_X,\mu_X)$,
  where $(X,d_X)$ is a complete separable metric space
  and $\mu_X$ a Borel probability measure on $X$.
  We sometimes say that $X$ is an mm-space, in which case
  the metric and the measure of $X$ are respectively indicated by
  $d_X$ and $\mu_X$.
\end{defn}

\begin{defn}[mm-Isomorphism]
  Two mm-spaces $X$ and $Y$ are said to be \emph{mm-isomorphic}
  to each other if there exists an isometry $f : \supp\mu_X \to \supp\mu_Y$
  such that $f_*\mu_X = \mu_Y$,
  where $f_*\mu_X$ is the push-forward of $\mu_X$ by $f$.
  Such an isometry $f$ is called an \emph{mm-isomorphism}.
  Denote by $\cX$ the set of mm-isomorphism classes of mm-spaces.
\end{defn}

Any mm-isomorphism between mm-spaces is automatically surjective,
even if we do not assume it.
Note that $X$ is mm-isomorphic to $(\supp\mu_X,d_X,\mu_X)$.

\emph{We assume that an mm-space $X$ satisfies
\[
X = \supp\mu_X
\]
unless otherwise stated.}

\begin{defn}[Lipschitz order] \label{defn:dom}
  Let $X$ and $Y$ be two mm-spaces.
  We say that $X$ (\emph{Lipschitz}) \emph{dominates} $Y$
  and write $Y \prec X$ if
  there exists a $1$-Lipschitz map $f : X \to Y$ with
  $f_*\mu_X = \mu_Y$.
  We call the relation $\prec$ on $\cX$ the \emph{Lipschitz order}.
\end{defn}

The Lipschitz order $\prec$ is a partial order relation on $\cX$

\subsection{Observable diameter}

The observable diameter is one of the most fundamental invariants
of an mm-space.

\begin{defn}[Partial and observable diameter] \label{defn:ObsDiam}
  Let $X$ be an mm-space.
  For a real number $\alpha$, we define
  the \emph{partial diameter
    $\diam(X;\alpha) = \diam(\mu_X;\alpha)$ of $X$}
  to be the infimum of $\diam(A)$,
  where $A \subset X$ runs over all Borel subsets
  with $\mu_X(A) \ge \alpha$ and $\diam(A)$ denotes the diameter of $A$.
  For a real number $\kappa > 0$, we define
  the \emph{observable diameter of $X$} to be
  \begin{align*}
    \ObsDiam(X;-\kappa) &:= \sup\{\;\diam(f_*\mu_X;1-\kappa) \mid\\
    &\qquad\qquad\text{$f : X \to \R$ is $1$-Lipschitz continuous}\;\}.
  \end{align*}
\end{defn}

The observable diameter is an invariant under mm-isomorphism.
Clearly, $\ObsDiam(X;-\kappa)$ is monotone nonincreasing in $\kappa > 0$.
Note that $\ObsDiam(X;-\kappa) = \diam(X;1-\kappa) = 0$ for $\kappa \ge 1$.

\begin{defn}[L\'evy family]
  A sequence of mm-spaces $X_n$, $n=1,2,\dots$,
  is called a \emph{L\'evy family} if
  \[
  \lim_{n\to\infty} \ObsDiam(X_n;-\kappa) = 0
  \]
  for any $\kappa > 0$.
\end{defn}

For an mm-space $X$ and a real number $t > 0$, we define
$tX$ to be the mm-space $X$ with the scaled metric $d_{tX} := t d_X$.

\begin{prop}
  Let $X$ be an mm-space.
  Then we have
  \[
  \ObsDiam(tX;-\kappa) = t \ObsDiam(X;-\kappa)
  \]
  for any $t,\kappa > 0$.
\end{prop}

\begin{prop} \label{prop:ObsDiam-dom}
  If $X \prec Y$, then
  \[
  \ObsDiam(X;-\kappa) \le \ObsDiam(Y;-\kappa)
  \]
  for any $\kappa > 0$.
\end{prop}

\subsection{Separation distance}

\begin{defn}[Separation distance] \label{defn:Sep}
  Let $X$ be an mm-space.
  For any real numbers $\kappa_0,\kappa_1,\cdots,\kappa_N > 0$
  with $N\geq 1$,
  we define the \emph{separation distance}
  \[
  \Sep(X;\kappa_0,\kappa_1, \cdots, \kappa_N)
  \]
  of $X$ as the supremum of $\min_{i\neq j} d_X(A_i,A_j)$
  over all sequences of $N+1$ Borel subsets $A_0,A_2, \cdots, A_N \subset X$
  satisfying that $\mu_X(A_i) \geq \kappa_i$ for all $i=0,1,\cdots,N$,
  where $d_X(A_i,A_j) := \inf_{x\in A_i,y\in A_j} d_X(x,y)$.
  If there exists no sequence $A_0,\dots,A_N \subset X$
  with $\mu_X(A_i) \ge \kappa_i$, $i=0,1,\cdots,N$, then
  we define
  \[
  \Sep(X;\kappa_0,\kappa_1, \cdots, \kappa_N) := 0.
  \]
\end{defn}

 We see that $\Sep(X;\kappa_0,\kappa_1, \cdots, \kappa_N)$ is
 monotone nonincreasing in $\kappa_i$ for each $i=0,1,\dots,N$.
 The separation distance is an invariant under mm-isomorphism.

\begin{prop}
  Let $X$ be an mm-space.
  Then we have
  \[
  \Sep(tX;\kappa_0,\kappa_1,\dots,\kappa_N)
  = t \Sep(X;\kappa_0,\kappa_1,\dots,\kappa_N)
  \]
  for any $t,\kappa_0,\kappa_1,\dots,\kappa_N > 0$.
\end{prop}

\begin{prop} \label{prop:Sep-prec}
  Let $X$ and $Y$ be two mm-spaces.
  If $X$ is dominated by $Y$, then we have,
  for any real numbers $\kappa_0,\dots,\kappa_N > 0$,
  \[
  \Sep(X;\kappa_0,\dots,\kappa_N) \le \Sep(Y;\kappa_0,\dots,\kappa_N).
  \]
\end{prop}

\begin{prop} \label{prop:ObsDiam-Sep}
  For any mm-space $X$ and any real numbers $\kappa$ and $\kappa'$
  with $\kappa > \kappa' > 0$, we have
  \begin{align}
    \tag{1} &\ObsDiam(X;-2\kappa) \le \Sep(X;\kappa,\kappa),\\
    \tag{2} &\Sep(X;\kappa,\kappa) \le \ObsDiam(X;-\kappa').
  \end{align}
\end{prop}

\subsection{Box distance and observable distance}

For a subset $A$ of a metric space $(X,d_X)$ and for a real number $r > 0$, we set
\[
U_r(A) := \{\;x \in X \mid d_X(x,A) < r\;\},
\]
where $d_X(x,A) := \inf_{a \in A} d_X(x,a)$.

\begin{defn}[Prokhorov distance]
  The \emph{Prokhorov distance $d_P(\mu,\nu)$ between two Borel probability
    measures $\mu$ and $\nu$ on a metric space $X$}
  is defined to be the infimum of $\varepsilon > 0$ satisfying
  \begin{equation} \label{eq:Proh}
    \mu(U_\varepsilon(A)) \ge \nu(A) - \varepsilon
  \end{equation}
  for any Borel subset $A \subset X$.
\end{defn}

The Prokhorov metric is a metrization of weak convergence of
Borel probability measures on $X$ provided that $X$ is a separable
metric space.

\begin{defn}[$\me$]
  Let $(X,\mu)$ be a measure space and $Y$ a metric space.
  For two $\mu$-measurable maps $f,g : X \to Y$, we define $\me_\mu(f,g)$
  to be the infimum of $\varepsilon \ge 0$ satisfying
  \begin{align} \label{eq:me}
    \mu(\{\;x \in X \mid d_Y(f(x),g(x)) > \varepsilon\;\}) \le \varepsilon.
  \end{align}
  We sometimes write $\me(f,g)$ by omitting $\mu$.
\end{defn}

$\me_\mu$ is a metric on the set of $\mu$-measurable maps from $X$ to $Y$
by identifying two maps if they are equal $\mu$-a.e.

\begin{lem} \label{lem:dP-me}
  Let $X$ be a topological space with a Borel probability measure $\mu$
  and $Y$ a metric space.
  For any two $\mu$-measurable maps $f,g : X \to Y$, we have
  \[
  d_P(f_*\mu,g_*\mu) \le \me_\mu(f,g).
  \]
\end{lem}

\begin{defn}[Parameter]
  Let $I := [\,0,1\,)$ and let $X$ be an mm-space.
  A map $\varphi : I \to X$ is called a \emph{parameter of $X$}
  if $\varphi$ is a Borel measurable map such that
  \[
  \varphi_*\cL^1 = \mu_X,
  \]
  where $\cL^1$ denotes the one-dimensional Lebesgue measure on $I$.
\end{defn}

Any mm-space has a parameter.

\begin{defn}[Box distance]
  We define the \emph{box distance $\square(X,Y)$ between
    two mm-spaces $X$ and $Y$} to be
  the infimum of $\varepsilon \ge 0$
  satisfying that there exist parameters
  $\varphi : I \to X$, $\psi : I \to Y$, and
  a Borel subset $I_0 \subset I$ such that
  \begin{align}
    & |\,\varphi^*d_X(s,t)-\psi^*d_Y(s,t)\,| \le \varepsilon
    \quad\text{for any $s,t \in I_0$};\tag{1}\\
    & \cL^1(I_0) \ge 1-\varepsilon,\tag{2}
  \end{align}
  where
  $\varphi^*d_X(s,t) := d_X(\varphi(s),\varphi(t))$ for $s,t \in I$.
\end{defn}

The box distance function $\square$ is a complete separable metric on $\cX$.

\begin{lem} \label{lem:box-dP}
  Let $X$ be a complete separable metric space.
  For any two Borel probability measures $\mu$ and $\nu$ on $X$,
  we have
  \[
  \square((X,\mu),(X,\nu)) \le 2 \, d_P(\mu,\nu).
  \]
\end{lem}

\begin{defn}[Observable distance $\dconc(X,Y)$] \label{defn:obs-dist}
  Denote by $\Lip_1(X)$ the set of $1$-Lipschitz continuous
  functions on an mm-space $X$.
  For any parameter $\varphi$ of $X$, we set
  \[
  \varphi^*\Lip_1(X)
  := \{\;f\circ\varphi \mid f \in \Lip_1(X)\;\}.
  \]
  We define the \emph{observable distance $\dconc(X,Y)$ between
    two mm-spaces $X$ and $Y$} by
  \[
  \dconc(X,Y) := \inf_{\varphi,\psi} d_H(\varphi^*\Lip_1(X),\psi^*\Lip_1(Y)),
  \]
  where $\varphi : I \to X$ and $\psi : I \to Y$ run over all parameters
  of $X$ and $Y$, respectively,
  and where $d_H$
  is the Hausdorff distance with respect to the metric $\me_{\cL^1}$.
  We say that a sequence of mm-spaces $X_n$, $n=1,2,\dots$,
  \emph{concentrates} to an mm-space $X$ if $X_n$ $\dconc$-converges to $X$
  as $n\to\infty$.
\end{defn}

\begin{prop}
  Let $\{X_n\}_{n=1}^\infty$ be a sequence of mm-spaces.
  Then, $\{X_n\}$ is a L\'evy family if and only if $X_n$ concentrates to
  a one-point mm-space as $n\to\infty$.
\end{prop}

\begin{prop} \label{prop:dconc-box}
  For any two mm-spaces $X$ and $Y$ we have
  \[
  \dconc(X,Y) \le \square(X,Y).
  \]
\end{prop}

\subsection{Pyramid}

\begin{defn}[Pyramid] \label{defn:pyramid}
  A subset $\cP \subset \cX$ is called a \emph{pyramid}
  if it satisfies the following (1), (2), and (3).
  \begin{enumerate}
  \item If $X \in \cP$ and if $Y \prec X$, then $Y \in \cP$.
  \item For any two mm-spaces $X, X' \in \cP$,
    there exists an mm-space $Y \in \cP$ such that
    $X \prec Y$ and $X' \prec Y$.
  \item $\cP$ is nonempty and $\square$-closed.
  \end{enumerate}
  We denote the set of pyramids by $\Pi$.

  For an mm-space $X$ we define
  \[
  \cP_X := \{\;X' \in \cX \mid X' \prec X\;\}.
  \]
  We call $\cP_X$ the \emph{pyramid associated with $X$}.
\end{defn}

It is trivial that $\cX$ is a pyramid.
Let $*$ denotes a one-point mm-space, i.e.,
an mm-space consists of a single point.
Then we see $\cP_* = \{*\}$.

In Gromov's book \cite{Gmv:greenbook}, the definition of a pyramid
is only by (1) and (2) of Definition \ref{defn:pyramid}.
We here put (3) as an additional condition for the Hausdorff property
of $\Pi$.

\begin{defn}
  We say that a sequence of mm-spaces $X_n$, $n=1,2,\dots$,
  \emph{approximates a pyramid $\cP$}
  if we have
  \[
  X_1 \prec X_2 \prec \dots \prec X_n \prec \cdots
  \quad\text{and}\quad
  \overline{\bigcup_{n=1}^\infty X_n}^\square = \cP,
  \]
  where the bar with $\square$ means the closure with respect to $\square$.
\end{defn}

\begin{defn}[Weak convergence] \label{defn:w-conv}
  Let $\cP_n, \cP \in \Pi$, $n=1,2,\dots$.
  We say that \emph{$\cP_n$ converges weakly to $\cP$}
  as $n\to\infty$
  if the following (1) and (2) are both satisfied.
  \begin{enumerate}
  \item For any mm-space $X \in \cP$, we have
    \[
    \lim_{n\to\infty} \square(X,\cP_n) = 0.
    \]
  \item For any mm-space $X \in \cX \setminus \cP$, we have
    \[
    \liminf_{n\to\infty} \square(X,\cP_n) > 0.
    \]
  \end{enumerate}
\end{defn}

\begin{thm}[Gromov, Shioya \cites{Gmv:greenbook, Sy:book, Sy:mmlim}]
  There exists a metric $\rho$ on $\Pi$
  satisfying the following {\rm(1)}--{\rm(4)}.
  \begin{enumerate}
  \item The metric $\rho$ is a metrization of weak convergence.
  \item The metric space $(\Pi,\rho)$ is compact.
  \item The map $\cX \ni X \mapsto \cP_X \in \Pi$ is a topological embedding
    with respect to $\dconc$ and $\rho$, and its image is dense in $\Pi$.
    In particular, $\Pi$ is a compactification of $(\cX,\dconc)$.
  \item For any two mm-spaces $X$ and $Y$, we have
    \[
    \rho(\cP_X,\cP_Y) \le \dconc(X,Y).
    \]
  \end{enumerate}
\end{thm}

\subsection{Measurement}

\begin{defn}[$\cM(N)$, $\cM(N,R)$, $\cX(N,R)$]
  Let $N$ be a natural number and $R$ a nonnegative real number.
  Denote by $\cM(N)$ the set of Borel probability measures on $\R^N$
  equipped with the Prokhorov metric $d_P$, and set
  \[
  \cM(N,R) := \{\;\mu \in \cM(N) \mid \supp\mu \subset B^N_R\;\},
  \]
  where $B^N_R := \{\;x \in \R^N \mid \|x\|_\infty \le R\;\}$
  and $\|\cdot\|_\infty$ denotes the $l_\infty$-norm on $\R^N$.
  We define
  \[
  \cX(N,R) := \{\;(B^N_R,\|\cdot\|_\infty,\mu) \mid \mu \in \cM(N,R)\;\}.
  \]
\end{defn}

Note that $\cM(N,R)$ and $\cX(N,R)$ are compact
with respect to $d_P$ and $\square$, respectively.

\begin{defn}[$N$-Measurement]
  Let $\cP$ be a pyramid, $N$ a natural number,
  and $R$ a nonnegative real number.
  We define
  \begin{align*}
    \cM(\cP;N) &:= \{\;\mu \in \cM(N) \mid (\R^N,\|\cdot\|_\infty,\mu) \in \cP\;\},\\
    \cM(\cP;N,R) &:= \cM(\cP;N) \cap \cM(N,R).
  \end{align*}
  We call $\cM(\cP;N)$ (resp.~$\cM(\cP;N,R)$)
  the \emph{$N$-measurement} (resp.~\emph{$(N,R)$-measurement}) \emph{of}
  $\cP$.  For an mm-space $X$, we define
  $\cM(X;N) := \cM(\cP_X;N)$ and $\cM(X;N,R) := \cM(\cP_X;N,R)$.
\end{defn}

The $N$-measurement $\cM(\cP;N)$ is a closed subset of $\cM(N)$
and the $(N,R)$-measurement $\cM(\cP;N,R)$ is a compact subset of $\cM(N)$.
We see
\[
\cM(X;N) := \{\;\Phi_*\mu_X \mid \Phi : X \to (\R^N,\|\cdot\|_\infty)
\ \text{is $1$-Lipschitz}\;\}.
\]

\begin{lem}[\cite{Sy:mmlim}*{Lemma 3.6}, \cite{Sy:book}*{Lemma 5.15}]
  \label{lem:cM-dconc}
  Let $X$ and $Y$ be two mm-spaces.
  For any natural number $N$ we have
  \[
  d_H(\cM(X;N),\cM(Y;N)) \le N \dconc(X,Y),
  \]
  where $d_H$ is the Hausdorff distance with respect to
  the Prohorov metric $d_P$.
\end{lem}

\begin{lem}[\cite{Sy:book}*{Lemma 5.41}] \label{lem:cM-NR}
  Let $\cP$ and $\cP'$ be two pyramids.
  For any natural number $N$ and nonnegative real number $R$,
  we have
  \[
  d_H(\cM(\cP;N,R),\cM(\cP';N,R)) \le 2\, d_H(\cM(\cP;N),\cM(\cP';N)).
  \]
\end{lem}

\section{Obsrevable Diameter for Pyramid} \label{sec:ObsDiam}

\begin{lem} \label{lem:rconti-ObsDiam}
  Let $X$ be an mm-space.
  Then we have the following (1) and (2).
  \begin{enumerate}
  \item The partial diameter $\diam(\mu_X;1-\kappa)$ is right-continuous
  in $\kappa > 0$.
  \item The observable diameter $\ObsDiam(X;-\kappa)$
    is right-continuous in $\kappa > 0$.
  \end{enumerate}
\end{lem}

\begin{proof}
  We prove (1).
  Let $\{\delta_n\}_{n=1}^\infty$ be a monotone decreasing sequence of
  positive real numbers converging to zero.
  Then, $\diam(\mu_{X}; 1 -(\kappa + \delta_n))$ is monotone nondecreasing in $n$
  and bounded from above by $\diam(\mu_{X}; 1 -\kappa)$.
  Set
  \[
  \alpha := \lim_{n \to \infty} \diam(\mu_{X}; 1 -(\kappa + \delta_n)).
  \]
  It is clear that $\alpha \le \diam(\mu_{X}; 1-\kappa)$.
  For (1), it suffices to prove $\diam(\mu_{X}; 1-\kappa) \le \alpha$.
  There are closed subsets $A_n \subset X$, $n=1,2,\dots$,
  with the property that
  $\mu_{X}(A_n) \ge 1 - (\kappa + \delta_n)$ for any $n$ and
  \[
  \lim_{n\to\infty} \diam(A_n) = \alpha.
  \]
  We take a monotone decreasing sequence $\{ \eta_p\}_{p=1}^{\infty}$
  of positive real numbers converging to zero.
  The inner regularity of $\mu_X$ proves that there are compact subsets
  $K_p \subset X$, $p=1,2,\dots$, such that $\mu_X(K_p) > 1 - \eta_p$
  and $K_p \subset K_{p+1}$ for every $p$.
  We have $\mu_X(A_n \cap K_p) > 1-(\kappa+\delta_n+\eta_p)$.
  Note that the set of closed subsets in $K_p$ is compact with respect to
  the Hausdorff distance.
  Thus, there is a Hausdorff convergent subsequence of
  $\{A_n \cap K_p\}_{n=1}^\infty$ for each $p$.
  By a diagonal argument, we find a common subsequence $\{n(m)\}$ of $\{n\}$
  in such a way that $A_{n(m)} \cap K_p$ Hausdorff converges as $m\to\infty$
  for any $p$.  Denote its limit by $A_p$.
  $\{A_p\}_{p=1}^\infty$ is a monotone nondecreasing sequence of compact
  subsets of $X$ satisfying that $\mu_{X}(A_p) \ge 1-(\kappa+\eta_p)$ for any $p$.
  Setting
  \[
  A := \bigcup_{p=1}^{\infty} A_p,
  \]
  we have $\mu_X(A) \ge 1-\kappa$.
  Since $\limsup_{n\to\infty} \diam(A_n) \le \alpha$,
  we obtain
  \[
  \diam(\mu_{X}; 1-\kappa) \le \diam(A) \le \alpha.
  \] 
  This completes the proof of (1).

  We prove (2).
  Since $\ObsDiam(X; -(\kappa + \delta))$ is monotone nonincreasing in $\delta$,
  we have
  \[
  \lim_{\delta\to 0+} \ObsDiam(X; -(\kappa + \delta)) \le \ObsDiam(X; -\kappa).
  \]
  By (1),
  \begin{align*}
    \ObsDiam(X; -\kappa) 
    & = \sup_{f\in\Lip_1(X)} \diam(f_*\mu_X; 1-\kappa) \\
    & = \sup_{f\in\Lip_1(X)} \lim_{\delta\to 0+} \diam(f_*\mu_X; 1-(\kappa + \delta)) \\
    & \le \lim_{\delta\to 0+} \ObsDiam(X; -(\kappa + \delta)).
  \end{align*}
  This completes the proof of the lemma.
\end{proof}

Note that $\diam(\mu_X;1-\kappa)$ and $\ObsDiam(X;-\kappa)$
are not necessarily left-continuous in $\kappa$,
e.g.,~for a discrete space.

\begin{defn}[Observable diameter of pyramid] \label{defn:ObsDiam-pyramid}
  Let $\kappa > 0$.
  The \emph{$\kappa$-observable diameter of a pyramid $\cP$} is defined to be
  \[
  \ObsDiam(\cP;-\kappa)
  := \lim_{\delta\to 0+} \sup_{X \in \cP} \ObsDiam(X;-(\kappa+\delta))
  \quad (\le +\infty).
  \]
\end{defn}

Note that $\sup_{X \in \cP} \ObsDiam(X;-(\kappa+\delta))$ is monotone
nonincreasing in $\delta$, so that the above limit always exists.
It follows from Definition \ref{defn:ObsDiam-pyramid} that
$\ObsDiam(\cP;-\kappa)$ is right-continuous in $\kappa > 0$.

The following means the consistency of the definition.

\begin{prop}
  For any mm-space $X$ we have
  \[
  \ObsDiam(\cP_X;-\kappa) = \ObsDiam(X;-\kappa)
  \]
  for any $\kappa > 0$.
\end{prop}

\begin{proof}
  The proposition follows from Proposition \ref{prop:ObsDiam-dom} and
  Lemma \ref{lem:rconti-ObsDiam}.
\end{proof}

For a pyramid $\cP$ and a real number $t > 0$,
we define
\[
t\cP := \{\;tX \mid X \in \cP\;\}.
\]

The following proposition is obvious.

\begin{prop} \label{prop:ObsDiam-scale}
  Let $\cP$ be a pyramid.
  Then we have
  \[
  \ObsDiam(t\cP;-\kappa) = t \ObsDiam(\cP;-\kappa)
  \]
  for any $t,\kappa > 0$.
\end{prop}

\begin{defn}[$\rho_R$]
  For two pyramids $\cP$, $\cP'$, and for a positive real number $R$, we define
  \[
  \rho_R(\cP,\cP')
  := \sum_{N=1}^\infty \frac{1}{N\, 2^{N+1}} d_H(\cM(\cP;N,NR),\cM(\cP',N,NR)),
  \]
  where $d_H$ is the Hausdorff distance with respect to the Prokhorov metric.
\end{defn}

\begin{lem} \label{lem:rhoR}
  Let $\cP$ and $\cP_n$, $n=1,2,\dots$, be pyramids.
  Then the following {\rm(1), (2),} and {\rm(3)} are all equivalent to each other.
  \begin{enumerate}
  \item $\cP_n$ converges weakly to $\cP$ as $n\to\infty$.
  \item $\cP_n \cap \cX(N,R)$ Hausdorff converges to $\cP \cap \cX(N,R)$
    as $n\to\infty$
    for any natural number $N$ and any nonnegative real number $R$,
    where the Hausdorff distance is induced from $\square$.
  \item $\cM(\cP_n;N,R)$ Hausdorff converges to $\cM(\cP;N,R)$ as $n\to\infty$
    for any natural number $N$ and any nonnegative real number $R$,
    where the Hausdorff distance is induced from $d_P$.
  \end{enumerate}
\end{lem}

\begin{proof}
  `(1) $\Longleftrightarrow$ (2)' follows from
  \cite{Sy:book}*{Lemma 6.18}. 
  
  `(3) $\implies$ (2)' follows from
  \[
  d_H(\cP\cap\cX(N,R),\cP'\cap\cX(N,R)) \le 2 d_H(\cM(\cP;N,R),\cM(\cP';N,R)),
  \]
  which is implied by Lemma \ref{lem:box-dP}.

  `(2) $\implies$ (3)' follows from \cite{Sy:book}*{Lemma 7.23}.
  This completes the proof.
\end{proof}

\begin{thm} \label{thm:rhoR}
  We have the following {\rm(1)}, {\rm(2)}, and {\rm(3)}.
  \begin{enumerate}
  \item $\rho_R$ for each $R > 0$ is a metric on $\Pi$ compatible with
    weak convergence.
  \item For any two pyramids $\cP$, $\cP'$, for any natural number $N$,
    and for any positive real number $R$, we have
    \[
    d_H(\cM(\cP;N,NR),\cM(\cP';N,NR)) \le N\, 2^{N+1} \rho_R(\cP,\cP').
    \]
  \item For any two mm-spaces $X$, $Y$, and any positive real number $R$,
    \[
    \rho_R(\cP_X,\cP_Y) \le \dconc(X,Y).
    \]
  \end{enumerate}
\end{thm}

\begin{proof}
  (1) follows from Lemma \ref{lem:rhoR}.

  (2) is obvious.
  
  By Lemmas \ref{lem:cM-dconc} and \ref{lem:cM-NR}, we have
  \begin{align*}
  d_H(\cM(X;N,NR),\cM(Y;N,NR))
  &\le 2\, d_H(\cM(X;N),\cM(Y;N))\\
  &\le 2N \dconc(X,Y),
  \end{align*}
  which implies (3).
  This completes the proof.
\end{proof}

\begin{lem} \label{lem:dP-diam}
 Let $\mu$ and $\nu$ be two Borel probability measures on $\R$
 and $\varepsilon$ a positive real number.
 If $d_P(\mu,\nu) < \varepsilon$, then
 \[
 \diam(\mu;1-(\kappa+\varepsilon)) \le \diam(\nu;1-\kappa) + 2\varepsilon.
 \]
 for any $\kappa > 0$.
\end{lem}

\begin{proof}
  Since $\mu(U_\varepsilon(A)) \ge \nu(A) - \varepsilon$ for any Borel subset
  $A \subset \R$, we have
  \begin{align*}
    \diam(\nu;1-\kappa)
    &= \inf\{\;\diam(A) \mid \nu(A) \ge 1-\kappa\;\}\\
    &\ge \inf\{\;\diam(A) \mid \mu(U_\varepsilon(A)) \ge 1-\kappa-\varepsilon\;\} \\
    &\ge \inf\{\;\diam(B) \mid \mu(B) \ge 1-\kappa-\varepsilon\;\} - 2\varepsilon \\
    &= \diam(\mu;1-\kappa-\varepsilon) -2\varepsilon.
  \end{align*}
\end{proof}

\begin{defn}[$\diam^D$, $\ObsDiam^D$]
  For an mm-space $X$ and a nonnegative real number $D$, we set
  \begin{align*}
    \diam^D(\mu_X;1-\kappa) &:= \min\{\;\diam(\mu_X;1-\kappa),D\;\},\\
    \ObsDiam^D(X;-\kappa) &:= \min\{\;\ObsDiam(X;-\kappa),D\;\}.
  \end{align*}
\end{defn}

\begin{lem} \label{lem:ObsDiam-2R}
  Let $\cP$ and $\cP'$ be two pyramids.
  If we have
  \[
  \cM(\cP;1,R) \subset U_\varepsilon(\cM(\cP';1,R))
  \]
  for two positive real numbers $\varepsilon$ and $R$, then
  \[
  \ObsDiam^{2R}(\cP;-(\kappa+\varepsilon))
  \le \ObsDiam^{2R}(\cP';-\kappa) + 2\varepsilon
  \]
  for any $\kappa > 0$.
\end{lem}

\begin{proof}
  By Lemma \ref{lem:dP-diam}, we have
  \begin{align*}
    &\ObsDiam^{2R}(\cP;-(\kappa+\varepsilon))\\
    &= \lim_{\delta\to 0+} \sup_{X \in \cP,\; f \in \Lip_1(X)}
    \diam^{2R}(f_*\mu_X;-(\kappa+\varepsilon+\delta))\\
    &= \lim_{\delta\to 0+} \sup_{X \in \cP,\; f \in \Lip_1(X),\; f(X) \subset [\,-R,R\,]}
    \diam(f_*\mu_X;-(\kappa+\varepsilon+\delta))\\
    &= \lim_{\delta\to 0+} \sup_{\mu \in \cM(\cP;1,R)}
    \diam(\mu;-(\kappa+\varepsilon+\delta))\\
    &\le \lim_{\delta\to 0+} \sup_{\mu' \in \cM(\cP';1,R)}
    \diam(\mu';-(\kappa+\delta)) + 2\varepsilon\\
    &= \ObsDiam^{2R}(\cP';-\kappa) + 2\varepsilon.
  \end{align*}
  This completes the proof.
\end{proof}

\begin{cor} \label{cor:ObsDiam-2R}
  Let $\cP$ and $\cP'$ be two pyramids.
  If $\rho_R(\cP,\cP') < \varepsilon/4$ for two real numbers $\varepsilon, R > 0$, then
  \[
  \ObsDiam^{2R}(\cP;-(\kappa+\varepsilon))
  \le \ObsDiam^{2R}(\cP';-\kappa) + 2\varepsilon
  \]
  for any $\kappa > 0$.
\end{cor}

\begin{proof}
  The corollary follows from Theorem \ref{thm:rhoR}(2)
  and Lemma \ref{lem:ObsDiam-2R}.
\end{proof}

\begin{thm}[Limit formula for observable diameter] \label{thm:lim-ObsDiam}
  Let $\cP$ and $\cP_n$, $n=1,2,\dots$, be pyramids.
  If $\cP_n$ converges weakly to $\cP$ as $n\to\infty$, then
  \begin{align*}
    \ObsDiam(\cP;-\kappa)
    &= \lim_{\varepsilon\to 0+} \liminf_{n\to\infty}
    \ObsDiam(\cP_n;-(\kappa+\varepsilon)) \\
    &= \lim_{\varepsilon\to 0+} \limsup_{n\to\infty}
    \ObsDiam(\cP_n;-(\kappa+\varepsilon))
  \end{align*}
  for any $\kappa > 0$.
\end{thm}

\begin{proof}
  Let $\kappa, R > 0$ be any two fixed numbers.
  For any real number $\varepsilon$ with $0 < \varepsilon < \kappa$,
  there is a number $n_0$ such that
  $\rho_R(\cP_n,\cP) < \varepsilon/2$ for any $n \ge n_0$.
  Let $n \ge n_0$.
  Corollary \ref{cor:ObsDiam-2R} implies
  \begin{align*}
   &\ObsDiam^{2R}(\cP_n;-(\kappa+\varepsilon))
   \le \ObsDiam^{2R}(\cP;-\kappa) + 2\varepsilon,\\
   &\ObsDiam^{2R}(\cP;-(\kappa+2\varepsilon))
   \le \ObsDiam^{2R}(\cP_n;-(\kappa+\varepsilon)) + 2\varepsilon.
  \end{align*}
  Taking the limits of these two inequalities as $n\to\infty$
  and then $\varepsilon \to 0+$, we have
  \begin{align*}
    \ObsDiam^{2R}(\cP;-\kappa)
    &= \lim_{\varepsilon\to 0+} \liminf_{n\to\infty}
    \ObsDiam^{2R}(\cP_n;-(\kappa+\varepsilon)) \\
    &= \lim_{\varepsilon\to 0+} \limsup_{n\to\infty}
    \ObsDiam^{2R}(\cP_n;-(\kappa+\varepsilon)).
  \end{align*}
  Since this holds for any $R > 0$,
  the proof is completed.
\end{proof}

\begin{ex}
  In our previous paper \cite{Sy:mmlim}*{Corollary 5.8},
  we obtain
  \[
  \lim_{n\to\infty} \ObsDiam(S^n(r_n);-\kappa)
  = \diam(\gamma_{\lambda^2}^1;1-\kappa)
  = 2\lambda I^{-1}((1-\kappa)/2)
  \]
  for any sequence of positive real numbers $r_n$, $n=1,2,\dots$,
  with $r_n/\sqrt{n} \to \lambda$,
  and for any $\kappa$ with $0 < \kappa < 1$, where
  $\gamma_{\lambda^2}^1$ denotes the one-dimensional
  centered Guassian measure on $\R$ with variance $\lambda^2$
  and $I(r) := \gamma^1[\,0,r\,]$ for $r \ge 0$.
  Since $\cP_{S^n(r_n)}$ converges weakly to the virtual infinite-dimensional
  Gaussian space $\Gamma_{\lambda^2}^\infty$ with variance $\lambda^2$
  (see \cite{Sy:mmlim}),
  we have, by Theorem \ref{thm:lim-ObsDiam},
  \[
  \ObsDiam(\cP_{\Gamma_{\lambda^2}^\infty};-\kappa)
  = \diam(\gamma_{\lambda^2}^1;1-\kappa)
  = 2\lambda I^{-1}((1-\kappa)/2)
  \]
  for any $\kappa$ and $\lambda$ with $0 < \kappa < 1$ and $\lambda \ge 0$.
\end{ex}

\section{Separation Distance for Pyramid} \label{sec:Sep}

\begin{lem} \label{lem:lconti-Sep}
  Let $X$ be an mm-space.
  Then we have
  \[
  \lim_{\delta\to 0+} \Sep(X; \kappa_0-\delta, \kappa_1-\delta, \dots,\kappa_N-\delta)
  = \Sep(X; \kappa_0, \kappa_1, \ldots, \kappa_N)
  \]
  for any $\kappa_0,\kappa_1,\dots,\kappa_N > 0$ with $N \ge 1$.
\end{lem}

\begin{proof}
  Let $\{\delta_n\}_{n=1}^{\infty}$ be a monotone decreasing sequence
  of positive real numbers converging to zero.
  Then, $\Sep(X; \kappa_0-\delta_n,\dots, \kappa_N-\delta_n)$
  is monotone nonincreasing in $n$.  We set
  \[
  \beta := \lim_{n \to \infty} \Sep(X; \kappa_0-\delta_n,\dots, \kappa_N-\delta_n).
  \]
  Since $\Sep(X; \kappa_0-\delta_n,\dots,\kappa_N-\delta_n) \ge
  \Sep(X; \kappa_0,\dots, \kappa_N)$, we have
  $\beta \ge \Sep(X; \kappa_0, \kappa_1, \ldots, \kappa_N)$. 
  It suffices to prove $\Sep(X; \kappa_0,\dots, \kappa_N) \ge \beta$.
  It follows from the definition of $\beta$ that
  there are Borel subsets $A_0^n, A_1^n, \ldots, A_N^n \subset X$
  such that $\mu_{X}(A_i^n) \ge \kappa_i - \delta_n$ for any $n$ and
  $i=0, 1, \dots, N$,
  and
  \[
  \lim_{n\to\infty} \min_{i \neq j}d_{X}(A_i^n, A_j^n) = \beta.
  \]
  We may assume that each $A_i^n$ is a closed set.
  Take a monotone decreasing sequence $\{ \eta_p \}_{p=1}^{\infty}$
  of positive real numbers converging to zero.
  By the inner regularity of $\mu_{X}$, there is a monotone nondecreasing
  sequence of compact subsets $K_p \subset X$, $p=1,2,\dots$, such that
  $\mu_{X}(K_p) > 1-\eta_p$ for any $p$.
  Set
  \[
  A_{i, p}^n := A_i^n \cap K_p.
  \]
  Each $A_{i, p}^n$ is a compact set and satisfies
  $\mu_{X}(A_{i, p}^n) > \kappa_i - \delta_n -\eta_p > 0$.
  For each $i$ and each $p$, the sequence $\{A_{i, p}^n\}_n$ has
  a Hausdorff convergent subsequence.
  By a diagonal argument, there is a common subsequence $\{n(m)\}$ of $\{n\}$
  such that $A_{i, p}^{n(m)}$ Hausdorff converges to a compact subset of $X$,
  say $A_{i,p}$, for any $i$ and $p$.
  $A_{i,p}$ is monotone nondecreasing in $p$ and satisfies
  $\mu_{X}(A_{i, p}) \ge \kappa_i -  \eta_p$.
  Setting
  \[
  A_i := \bigcup_{p=1}^{\infty} A_{i, p},
  \]
  we have $\mu_{X}(A_i) \ge \kappa_i$.
  Since
  \[
  \min_{i \neq j} d_X(A_i^{n(m)}, A_j^{n(m)})
  \le \min_{i \neq j} d_X(A_{i, p}^{n(m)}, A_{j, p}^{n(m)})
  \]
  we obtain
  \[
  \beta \leq \min_{i \neq j}d_{X}(A_i, A_j) \leq \Sep(X; \kappa_0, \dots, \kappa_N).
  \]
  This completes the proof.
\end{proof}

\begin{rem}
  Lemma \ref{lem:lconti-Sep}
  and the monotonicity of $\Sep(X;\kappa_0,\dots,\kappa_N)$ in $\kappa_i$
  together imply that
  $\Sep(X;\kappa_0-\delta_0,\dots,\kappa_N-\delta_N)$
  converges to $\Sep(X;\kappa_0,\dots,\kappa_N)$
  as $\delta_0,\dots,\delta_N \to 0+$.
\end{rem}

\begin{defn}[Separation distance of pyramid] \label{defn:Sep-pyramid}
  For a pyramid $\cP$ and $\kappa_0,\dots,\kappa_N > 0$, we define
  \begin{align*}
    \Sep(\cP;\kappa_0,\kappa_1,\dots,\kappa_N)
    &:= \lim_{\delta\to 0+} \sup_{X \in \cP} \Sep(X;\kappa_0-\delta,\kappa_1-\delta,
    \dots,\kappa_N-\delta) \\
    (&\le +\infty).
  \end{align*}
\end{defn}

$\Sep(\cP;\kappa_0,\kappa_1,\dots,\kappa_N)$ is left-continuous
and monotone nonincreasing in $\kappa_0,\dots,\kappa_N$.

\begin{prop}
  For any mm-space $X$ we have
  \[
  \Sep(\cP_X;\kappa_0,\kappa_1,\dots,\kappa_N)
  = \Sep(X;\kappa_0,\kappa_1,\dots,\kappa_N)
  \]
  for any $\kappa_0,\dots,\kappa_N > 0$.
\end{prop}

\begin{proof}
  The proposition follows from
  Proposition \ref{prop:Sep-prec} and Lemma \ref{lem:lconti-Sep}.
\end{proof}

The following is obvious.

\begin{prop} \label{prop:Sep-scale}
  Let $\cP$ be a pyramid.
  Then we have
  \[
  \Sep(t\cP;\kappa_0,\kappa_1,\dots,\kappa_N)
  = t \Sep(\cP;\kappa_0,\kappa_1,\dots,\kappa_N)
  \]
  for any $t,\kappa_0,\kappa_1,\dots,\kappa_N > 0$.
\end{prop}

\begin{defn}[$\Sep^D$]
  For a pyramid $\cP$ and for positive real numbers $\kappa_0,\dots,\kappa_N$,
  and $D$,
  we set
  \[
  \Sep^D(\cP;\kappa_0,\kappa_1,\dots,\kappa_N)
  := \min\{\;\Sep(\cP;\kappa_0,\kappa_1,\dots,\kappa_N),D\;\}.
  \]
\end{defn}

\begin{lem} \label{lem:Sep-2R}
  Let $\cP$ and $\cP'$ be two pyramids.
  If we have
  \[
  \cM(\cP;N+1,R) \subset U_\varepsilon(\cM(\cP';N+1,R))
  \]
  for a natural number $N$ and for two real numbers $\varepsilon, R > 0$, then
  \[
  \Sep^{2R}(\cP;\kappa_0,\kappa_1,\dots,\kappa_N)
  \le \Sep^{2R}(\cP';\kappa_0-\varepsilon,\kappa_1-\varepsilon,\dots,
  \kappa_N-\varepsilon) + 2\varepsilon
  \]
  for any $\kappa_0,\dots,\kappa_N > \varepsilon$.
\end{lem}

\begin{proof}
  We take any $\delta > 0$ and any mm-space $X \in \cP$.
  Let $0 < r < \Sep^{2R}(X;\kappa_0-\delta,\dots,\kappa_N-\delta)$.
  There are Borel subsets $A_0,\dots,A_N \subset X$
  such that $\mu_X(A_i) \ge \kappa_i-\delta$
  and $d_X(A_i,A_j) \ge r$ for any different $i$ and $j$. 
  Set $f_i(x) := \min\{\;d_X(x,A_i),r\;\}$ for $x \in X$,
  $F := (f_0,\dots,f_N) : X \to \R^{N+1}$,
  and $F^R := (f_0-R,\dots,f_N-R) : X \to \R^{N+1}$.
  $F^R_*\mu_X$ belongs to $\cM(\cP;N+1,R)$.
  By $\cM(\cP;N+1,R) \subset U_\varepsilon(\cM(\cP';N+1,R))$,
  there are an mm-space $Y \in \cP'$ and
  a $1$-Lipschitz map $G^R : Y \to (\R^{N+1},\|\cdot\|_\infty)$ such that
  $d_P(F^R_*\mu_X,G^R_*\mu_Y) < \varepsilon$.
  We find maps $g_i : Y \to \R$, $i=0,1,\dots,N$, in such a way that
  $(g_0-R,\dots,g_N-R) = G^R$.  Setting $G := (g_0,\dots,g_N)$
  we have $d_P(F_*\mu_X,G_*\mu_Y) < \varepsilon$.
  Let
  \[
  B_i := \{\;g_i < \varepsilon, \ g_j > r-\varepsilon\ \text{for any $j \neq i$}\;\} \subset Y.
  \]
  We see that, for $i=0,1,\dots,N$,
  \begin{align*}
    \mu_Y(B_i) &= G_*\mu_Y( x_i < \varepsilon,\ x_j > r-\varepsilon
    \ \text{for any $j \neq i$}) \\
    &= G_*\mu_Y(U_\varepsilon(x_i \le 0,\ x_j \ge r\ \text{for any $j \neq i$})) \\
    &\ge F_*\mu_X(x_i \le 0,\ x_j \ge r\ \text{for any $j \neq i$}) - \varepsilon \\
    &= \mu_X(A_i)-\varepsilon \ge \kappa_i-\varepsilon-\delta.
  \end{align*}
  For any $y \in B_i$ and $y' \in B_j$ with $i\neq j$, we have
  \[
  d_Y(y,y') \ge |g_i(y)-g_i(y')| > r-2\varepsilon
  \]
  and hence $d_Y(B_i,B_j) \ge r-2\varepsilon$.
  Thus,
  \[
  \Sep(Y;\kappa_0-\varepsilon-\delta,\dots,\kappa_N-\varepsilon-\delta)
  \ge r-2\varepsilon.
  \]
  By the arbitrariness of $r$,
  \[
  \Sep^{2R}(X;\kappa_0-\delta,\dots,\kappa_N-\delta)
  \le \Sep(Y;\kappa_0-\varepsilon-\delta,\dots,\kappa_N-\varepsilon-\delta)
  +2\varepsilon.
  \] 
  This completes the proof.
\end{proof}

\begin{cor} \label{cor:Sep-2R}
  Let $\cP$ and $\cP'$ be two pyramids.
  If $\rho_R(\cP,\cP') < \frac{\varepsilon}{(N+1)2^{N+2}}$
  for a natural number $N$ and for two real numbers $\varepsilon,R > 0$, then
  \[
  \Sep^{2R}(\cP;\kappa_0,\kappa_1,\dots,\kappa_N)
  \le \Sep^{2R}(\cP';\kappa_0-\varepsilon,\kappa_1-\varepsilon,\dots,
  \kappa_N-\varepsilon) + 2\varepsilon
  \]
  for any $\kappa_0,\dots,\kappa_N > \varepsilon$.
\end{cor}

\begin{proof}
  The corollary follows from Theorem \ref{thm:rhoR}(2)
  and Lemma \ref{lem:Sep-2R}.
\end{proof}

\begin{thm}[Limit formula for separation distance] \label{thm:lim-Sep}
  Let $\cP$ and $\cP_n$, $n=1,2,\dots$, be pyramids.
  If $\cP_n$ converges weakly to $\cP$ as $n\to\infty$, then
  \begin{align*}
    \Sep(\cP;\kappa_0,\kappa_1,\dots,\kappa_N)
    &= \lim_{\varepsilon\to 0+} \liminf_{n\to\infty}
    \Sep(\cP_n;\kappa_0-\varepsilon,\kappa_1-\varepsilon,\dots,\kappa_N-\varepsilon) \\
    &= \lim_{\varepsilon\to 0+} \limsup_{n\to\infty}
    \Sep(\cP_n;\kappa_0-\varepsilon,\kappa_1-\varepsilon,\dots,\kappa_N-\varepsilon)
  \end{align*}
  for any $\kappa_0,\dots,\kappa_N > 0$.
\end{thm}

\begin{proof}
  The theorem is obtained in the same way as in the proof of Theorem
  \ref{thm:lim-ObsDiam}, by using Corollary \ref{cor:Sep-2R}.
\end{proof}

\begin{prop} \label{prop:ObsDiam-Sep-pyramid}
  Let $\cP$ be a pyramid.
  Then we have
  \begin{align*}
    \tag{1} &\ObsDiam(\cP;-2\kappa) \le \Sep(\cP;\kappa,\kappa) \\
    \tag{2} &\Sep(\cP;\kappa,\kappa) \le \ObsDiam(\cP;-\kappa')
  \end{align*}
  for any real number $\kappa$ and $\kappa'$ with $0 < \kappa' < \kappa$.
\end{prop}

\begin{proof}
  Let $0 < \kappa' < \kappa$.
  We take a sequence of mm-spaces $Y_n$, $n=1,2,\dots$,
  approximating $\cP$.
  Proposition \ref{prop:ObsDiam-Sep} implies that, for any $\delta$
  with $0 < \delta < \kappa$,
  \begin{align*}
    \ObsDiam(Y_n;-2(\kappa+\delta)) &\le \Sep(Y_n;\kappa+\delta,\kappa+\delta) \\
    &\le \Sep(Y_n;\kappa-\delta,\kappa-\delta).
  \end{align*}
  Since $\cP_{Y_n} \to \cP$ as $n\to\infty$,
  applying Theorems \ref{thm:lim-ObsDiam} and \ref{thm:lim-Sep}
  yields (1).
  (2) is proved in the same way.
  This completes the proof.
\end{proof}

\section{$N$-L\'evy Family} \label{sec:N-Levy}

\begin{defn}[$N$-L\'evy family] \label{defn:N-Levy}
  Let $N$ be a natural number.
  A sequence of pyramids $\cP_n$, $n=1,2,\dots$, is called
  an \emph{$N$-L\'evy family}
  if
  \[
  \lim_{n\to\infty} \Sep(\cP_n;\kappa_0,\kappa_1,\dots,\kappa_N) = 0
  \]
  for any $\kappa_0,\kappa_1,\dots,\kappa_N > 0$ with $\sum_{i=0}^N \kappa_i < 1$.
  A $1$-L\'evy family is called a \emph{L\'evy family}.
\end{defn}

\begin{defn}[$\#\cP$]
  For a pyramid $\cP$,
  we define
  \[
  \#\cP := \sup_{X \in \cP} \# X \quad (\le +\infty),
  \]
  where $\# X$ denotes the number of elements of $X$.
\end{defn}

\begin{lem} \label{lem:Sep-fin}
  Let $X$ be an mm-space, $\cP$ a pyramid, and $N$ a natural number.
  Then we have the following {\rm(1)} and {\rm(2)}.
  \begin{enumerate}
  \item We have
    \[
    \Sep(X;\kappa_0,\kappa_1,\dots,\kappa_N) = 0
    \]
    for any $\kappa_0,\kappa_1,\dots,\kappa_N > 0$
    with $\sum_{i=1}^N \kappa_i < 1$ if and only if $\# X \le N$.
  \item We have
    \[
    \Sep(\cP;\kappa_0,\kappa_1,\dots,\kappa_N) = 0
    \]
    for any $\kappa_0,\kappa_1,\dots,\kappa_N > 0$
    with $\sum_{i=1}^N \kappa_i < 1$ if and only if $\#\cP \le N$.
  \end{enumerate}
\end{lem}

\begin{proof}
  We prove (1).
  The `if' part is obvious.
  Let us prove the `only if' part.
  It suffices to show that, if $\#X \ge N+1$, then
  $\Sep(X;\kappa_0,\kappa_1,\dots,\kappa_N) > 0$
  for some $\kappa_0,\kappa_1,\dots,\kappa_N > 0$
    with $\sum_{i=1}^N \kappa_i < 1$.

  If $\#X \ge N+2$, then we find different $N+2$ points
  $x_0,x_1,\dots,x_{N+1} \in X$ and set
  \[
  r := \min_{i\neq j} d_X(x_i,x_j) > 0, \quad A_i := U_{r/3}(x_i),
  \quad\text{and}\quad \kappa_i := \mu_X(A_i).
  \]
  Note that each $\kappa_i$ is positive.
  We have $\min_{i\neq j} d_X(A_i,A_j) \ge r/3$ by triangle inequalities,
  and therefore
  \[
  \Sep(X;\kappa_0,\kappa_1,\dots,\kappa_N) \ge r/3 > 0.
  \]
  We also have
  $\sum_{i=0}^N \kappa_i \le 1-\kappa_{N+1} < 1$.
  
  If $\#X = N+1$, then we find real numbers $\kappa_0,\dots,\kappa_N$
  such that $0 < \kappa_i < \min_j \mu_X(\{x_j\})$ for any $i$,
  where $\{x_0,x_1,\dots,x_N\} := X$.
  We see $\sum_{i=0}^N \kappa_i < 1$.
  Since $x_0,x_1,\dots,x_N$ are different to each other,
  \[
  \Sep(X;\kappa_0,\kappa_1,\dots,\kappa_N) > 0.
  \]
  (1) has been proved.
  
  We prove (2).  The `if' part is easy to prove.
  We prove the `only if' part.
  Let $\kappa_0,\kappa_1,\dots,\kappa_N > 0$ be any real numbers
  with $\sum_{i=1}^N \kappa_i < 1$.
  We assume
  \[
  \Sep(\cP;\kappa_0,\kappa_1,\dots,\kappa_N) = 0.
  \]
  Then, for any mm-space $X \in \cP$, we have, by Lemma \ref{lem:lconti-Sep},
  \begin{align*}
    \Sep(X;\kappa_0,\dots,\kappa_N) 
    &= \lim_{\delta\to 0+} \Sep(X;\kappa_0-\delta,\dots,\kappa_N-\delta)\\
    &\le \lim_{\delta\to 0+} \sup_{Y\in\cP}
    \Sep(Y;\kappa_0-\delta,\dots,\kappa_N-\delta)\\
    &= \Sep(\cP;\kappa_0,\kappa_1,\dots,\kappa_N) = 0,
  \end{align*}
  which together with (1) implies $\#X \le N$.
  By the arbitrariness of $X$ we obtain $\#\cP \le N$.
  This completes the proof of the lemma.
\end{proof}

\begin{defn}[Extended mm-space]
  We consider to generalize the definition of an mm-space
  such that the metric is allowed to take values in $[\,0,+\infty\,]$.
  We call such a space an \emph{extended mm-space}.
  We define the Lipschitz order $\prec$ between extended mm-spaces
  in the same manner,
  and define the \emph{pyramid $\cP_X$ associated with an extended
  mm-space $X$} by
  \[
  \cP_X := \{\; X' \in \cX \mid X' \prec X \;\}.
  \]
\end{defn}

It is easy to see that $\cP_X$ is a pyramid for any extended mm-space $X$.
For an extended mm-space $X$ and a real number $D > 0$, we define
$X^D := (X,d_{X^D},\mu_X)$, where $d_{X^D}(x,y) := \min\{d_X(x,y),D\}$
for $x,y \in X$.  Then, $X^D$ is an mm-space belonging to $\cP_X$.
We observe that $\cP_X$ coincides with the $\square$-closure
of $\bigcup_{0 < D < +\infty} \cP_{X^D}$.

\begin{prop} \label{prop:fin-ex}
  We have $\#\cP < +\infty$ if and only if
  there exists a finite extended mm-space $X$ such that $\cP = \cP_X$.
  In this case, we have $\#X = \#\cP$.
\end{prop}

\begin{proof}
  Note that the number of elements $\#X$ is monotone nondecreasing
  in $X$ with respect to the Lipschitz order.
  We easily see that $\#X = \#\cP_X$ for any extended mm-space.
  In particular,  we have the `if' part of the proposition.
  We prove the `only if' part.
  Assume that $N := \#\cP < +\infty$.
  Let $\{X_n\}_{n=1}^\infty$ be a sequence of mm-spaces
  approximating $\cP$.
  There is a natural number $n_0$ such that $\#X_n = N$ for any $n \ge n_0$.
  Since $X_1 \prec X_2 \prec \dots \prec X_n \prec \dots$,
  we find $1$-Lipschitz maps $f_n : X_{n+1} \to X_n$ which pushes
  $\mu_{X_{n+1}}$ forward to $\mu_{X_n}$.
  For $n \ge n_0$, the map $f_n$ is bijective.
  Let $\{x_1^n,x_2^n,\dots,x_N^n\} := X_n$
  such that $f_n(x_i^{n+1}) = x_i^n$ for any $i=1,2,\dots,N$ and $n \ge n_0$.
  We see that $\mu_{X_n}(\{x_i^n\})$ is independent of $n \ge n_0$
  and that $d_{X_n}(x_i^n,x_j^n)$ is monotone nondecreasing in $n \ge n_0$
  for any $i,j=1,2,\dots,N$.
  Let $X = \{x_1,x_2,\dots,x_N\}$ be an $N$-point space
  and define an (extended) mm-structure of $X$ by
  \[
  d_X(x_i,x_j) := \lim_{n\to\infty} d_{X_n}(x_i^n,x_j^n) \ (\le +\infty)
  \quad\text{and}\quad
  \mu_X(\{x_i\}) := \mu_{X_n}(\{x_i^n\})
  \]
  for $i,j=1,2,\dots,N$ and $n\ge n_0$.
  Then we have $X_n \prec X$ for any $n$ and therefore
  $\cP \subset \cP_X$.
  The rest of the proof is to show $\cP_X \subset \cP$.
  Let $D > 0$ be any number.
  Since $\lim_{n\to\infty} d_{X_n^D}(x_i^n,x_j^n) = d_{X^D}(x_i,x_j)$
  for any $i,j = 1,2,\dots,N$,
  we see that $X_n^D$ $\square$-converges to $X^D$ as $n\to\infty$,
  which together with $X_n^D \in \cP$ implies $X^D \in \cP$.
  Since $\cP_X$ coincides with the $\square$-closure of
  $\bigcup_{0 < D <+\infty} \cP_{X^D}$, we have $\cP_X \subset \cP$.
  This completes the proof.
\end{proof}

Combining the statements proved before, we obtain the following theorem.

\begin{thm} \label{thm:N-Levy}
  Let $\{\cP_n\}_{n=1}^\infty$ be a sequence of pyramids
  converging weakly to a pyramid $\cP$, and $N$ a natural number.
  Then, the following {\rm(1)} and {\rm(2)} are equivalent to each other.
  \begin{enumerate}
  \item $\{\cP_n\}$ is an $N$-L\'evy family.
  \item There exists a finite extended mm-space $X$ with $\#X \le N$
    such that $\cP = \cP_X$.
  \end{enumerate}
\end{thm}

\begin{proof}
  Theorem \ref{thm:lim-Sep} and Lemma \ref{lem:Sep-fin}(2) together
  prove the equivalence
  between (1) and $\#\cP \le N$.
  These are also equivalent to (2) due to Proposition \ref{prop:fin-ex}.
  This completes the proof.
\end{proof}

For a compact (weighted) Riemannian manifold $M$,
we denote by $\lambda_N(M)$ the $N$-th nonzero eigenvalue
of the (weighted) Laplacian on $M$.
Since a sequence of compact (weighted) Riemannian manifolds $M_n$,
$n=1,2,\dots$,
is an $N$-L\'evy family if $\lambda_N(M_n) \to +\infty$ as $n\to\infty$
(see \cite{FnSy}*{Corollary 4.3}),
we have the following corollary to Theorem \ref{thm:N-Levy}.

\begin{cor}
  Let $\{M_n\}_{n=1}^\infty$ be a sequence of compact $($weighted$)$ Riemannian
  manifolds such that $\lambda_N(M_n) \to +\infty$ as $n\to\infty$
  for a natural number $N$.
  Then, there exist a subsequence $\{M_{n_i}\}$ of $\{M_n\}$ and
  a finite extended mm-space $X$ with $\#X \le N$
  such that $\cP_{M_{n_i}}$ converges weakly to $\cP_X$ as $i\to\infty$.
\end{cor}

Let $*$ denote a one-point mm-space.
Note that $\cP_* = \{*\}$.

\begin{cor} \label{cor:Levy}
  Let $\cP_n$, $n=1,2,\dots$, be pyramids.
  Then, the following {\rm(1)}, {\rm(2)}, and {\rm(3)} are equivalent to each other.
  \begin{enumerate}
  \item $\{\cP_n\}$ is a L\'evy family.
  \item $\cP_n$ converges weakly to $\cP_*$ as $n\to\infty$.
  \item $\lim_{n\to\infty} \ObsDiam(\cP_n;-\kappa) = 0$ for any $\kappa > 0$.
  \end{enumerate}
\end{cor}

\begin{proof}
  `(1) $\Longleftrightarrow$ (2)' follows from Theorem \ref{thm:N-Levy}.
  `(1) $\Longleftrightarrow$ (3)' follows from
  Proposition \ref{prop:ObsDiam-Sep-pyramid},
  where we use the monotonicity of $\Sep(\cP;\kappa_0,\kappa_1)$
  in $\kappa_0,\kappa_1$.
  This completes the proof.
\end{proof}

As an application of Theorem \ref{thm:N-Levy}, we prove the following corollary,
which is an extension of \cite{FnSy}*{Theorem 4.4}.
The proof here is much easier than that in \cite{FnSy}.

\begin{cor} \label{cor:shrink}
  Let $\{\cP_n\}_{n=1}^\infty$ be an $N$-L\'evy family of pyramids
  such that
  \begin{equation} \label{eq:shrink}
  \ObsDiam(\cP_n;-\kappa) < +\infty
  \end{equation}
  for any $\kappa$ with $0 < \kappa < 1$ and for any $n$.
  Then we have one of the following {\rm(1)} and {\rm(2)}.
  \begin{enumerate}
  \item $\{\cP_n\}$ is a L\'evy family.
  \item There is a subsequence $\{\cP_{n_i}\}_{i=1}^\infty$ of $\{\cP_n\}$
    and a sequence of real numbers $t_i$ with $0 < t_i \le 1$, $i=1,2,\dots$,
    such that $t_i\cP_{n_i}$ converges weakly to $\cP_X$
    for some finite mm-space $X$ with $2 \le \#X \le N$.
  \end{enumerate}
\end{cor}

Note that the observable diameter of an mm-space is always finite,
so that \eqref{eq:shrink} holds for any sequence of mm-spaces.

\begin{proof}
  Let $\{\cP_n\}_{n=1}^\infty$ be an $N$-L\'evy family of pyramids
  that is not a L\'evy family.
  Taking a subsequence of $\{\cP_n\}$, we assume that $\{\cP_n\}$
  converges weakly to a pyramid $\cP$.
  By Theorem \ref{thm:N-Levy},
  there is a finite extended mm-space $Y$ such that $\cP = \cP_Y$
  with $2 \le \#Y \le N$.
  Take a real number $\kappa$ with $0 < \kappa < \min_{y \in Y} \mu_Y(\{y\})$.
  If $\ObsDiam(\cP_n;-\kappa)$ is bounded from above for all $n$,
  then Theorem \ref{thm:lim-ObsDiam} implies
  the finiteness of $\ObsDiam(Y;-\kappa') = \diam(Y)$ for $0 < \kappa' < \kappa$,
  so that
  $Y$ is an mm-space and we have the theorem.
  Assume that $\ObsDiam(\cP_n;-\kappa)$ is unbounded.
  Replacing $\{\cP_n\}$ with a subsequence
  we assume that $\ObsDiam(\cP_n;-\kappa) \ge 1$ for any $n$.
  Setting $t_n := \ObsDiam(\cP_n;-\kappa)^{-1}$,
  we have $0 < t_n \le 1$ for any $n$ by \eqref{eq:shrink}.
  We replace $\{t_n\cP_n\}$ with a weakly convergent subsequence of it.
  By Theorem \ref{thm:N-Levy},
  there is a finite extended mm-space $X$ such that
  $t_n\cP_n$ converges weakly to $\cP_X$ as $n\to\infty$.
  Since $t_n\cP_n \subset \cP_n$, we have $\cP_X \subset \cP_Y$
  and so $X \prec Y$.
  In particular, $\kappa < \min_{x \in X} \mu_X(\{x\})$.
  Therefore, applying Theorem \ref{thm:lim-ObsDiam} yields
  \begin{align*}
    \diam(X) &= \ObsDiam(\cP_X;-\kappa)\\
    &= \lim_{\delta\to 0+} \liminf_{n\to\infty} \ObsDiam(t_n\cP_n;-(\kappa+\delta)) \\
    &= \lim_{\delta\to 0+} \liminf_{n\to\infty}
    \frac{\ObsDiam(\cP_n;-(\kappa+\delta))}{\ObsDiam(\cP_n;-\kappa)}
    \le 1,
  \end{align*}
  so that $X$ is an mm-space.
  Moreover, since $\ObsDiam(t_n\cP_n;-\kappa) = 1$,
  $\{t_n\cP_n\}$ is not a L\'evy family (see Corollary \ref{cor:Levy})
  and $X$ consists of at least two points.
  This completes the proof.
\end{proof}

\section{Dissipation and Phase Transition Property} \label{sec:phase}

\begin{defn}[Dissipation]
  Let $\{\cP_n\}_{n=1}^\infty$ be a sequence of pyramids
  and let $0 < \delta \le +\infty$.
  We say that $\{\cP_n\}$ \emph{$\delta$-dissipates} if
  \[
  \liminf_{n\to\infty} \Sep(\cP_n;\kappa_0,\kappa_1,\dots,\kappa_N) \ge \delta
  \]
  for any $\kappa_0,\kappa_1,\dots,\kappa_N > 0$ with $\sum_{i=0}^N \kappa_i < 1$.
  We say that $\{\cP_n\}$ \emph{weakly dissipates} if
  \[
  \liminf_{n\to\infty} \Sep(\cP_n;\kappa_0,\kappa_1,\dots,\kappa_N) > 0
  \]
  for any $\kappa_0,\kappa_1,\dots,\kappa_N > 0$ with $\sum_{i=0}^N \kappa_i < 1$.
\end{defn}

%

\begin{defn}[Dissipated pyramid]
  Let $\cP$ be a pyramid and let $0 < \delta \le +\infty$.
  $\cP$ is said to be \emph{$\delta$-dissipated}
  if $\cP$ contains all mm-spaces with diameter $\le \delta$.
   We say that $\cP$ is \emph{weakly dissipated} if
  \[
  \Sep(\cP;\kappa_0,\kappa_1,\dots,\kappa_N) > 0
  \]
  for any $\kappa_0,\kappa_1,\dots,\kappa_N > 0$ with $\sum_{i=0}^N \kappa_i < 1$.
\end{defn}

Note that $\cP$ is $\infty$-dissipated if and only if $\cP = \cX$.

\begin{lem} \label{lem:dissipate}
  Let $\{\cP_n\}_{n=1}^\infty$ be a sequence of pyramids
  converging weakly to a pyramid $\cP$, and let $0 < \delta \le +\infty$.
  Then, the following {\rm(1)} and {\rm(2)} are equivalent to each other.
  \begin{enumerate}
  \item $\{\cP_n\}$ $\delta$-dissipates {\rm(}resp.~weakly dissipates{\rm)}.
  \item $\cP$ is $\delta$-dissipated {\rm(}resp.~weakly dissipated{\rm)}.
  \end{enumerate}
\end{lem}

\begin{proof}
  The lemma for the weak dissipation follows from Theorem \ref{thm:lim-Sep}.

  We prove the lemma for the $\delta$-dissipation.
  Theorem \ref{thm:lim-Sep} implies that
  (1) is equivalent to the following:
  \begin{enumerate}
  \item[(3)] For any $\kappa_0,\dots,\kappa_N > 0$ with $\sum_{i=0}^N \kappa_i < 1$,
  we have
  \[
  \Sep(\cP;\kappa_0,\dots,\kappa_N) \ge \delta.
  \]
  \end{enumerate}
  There is a sequence of mm-spaces $X_n$, $n=1,2,\dots$,
  apprioximating $\cP$.
  (3) is equivalent to
  \[
  \liminf_{n\to\infty} \Sep(X_n;\kappa_0,\dots,\kappa_N) \ge \delta
  \]
  for any $\kappa_0,\dots,\kappa_N > 0$ with $\sum_{i=0}^N \kappa_i < 1$.
  Due to \cite{Sy:book}*{Proposition 8.5},
  this is equivalent to (2).
  The proof is completed.
\end{proof}

\begin{defn}[Phase transition property] \label{defn:phase}
  Let $\{\cP_n\}_{n=1}^\infty$ be a sequence of pyramids.
  We say that $\{\cP_n\}$ has the \emph{phase transition property}
  if there exists a sequence of positive real numbers $c_n$, $n=1,2,\dots$,
  satisfying the following (1) and (2).
  \begin{enumerate}
  \item For any sequence of positive numbers $t_n$ with
  $t_n/c_n \to 0$, the sequence $\{t_n\cP_n\}$ is a L\'evy family.
  \item For any sequence of positive numbers $t_n$ with
  $t_n/c_n \to +\infty$, the sequence $\{t_n\cP_n\}$ $\infty$-dissipates.
  \end{enumerate}
  We call such a sequence $\{c_n\}$ a sequence of \emph{critical scale order}.
  We say that a sequence of mm-spaces $X_n$ has
  the \emph{phase transition property}
  if so does $\{\cP_{X_n}\}$.
\end{defn}

\begin{rem}
  (1) of Definitnion \ref{defn:phase} is equivalent to
  \[
  \limsup_{n\to\infty} \ObsDiam(c_n\cP_n;-\kappa) < +\infty
  \]
  for any $\kappa$ with $0 < \kappa < 1$.
  (2) is equivalent to the weak dissipation property of $\{c_n\cP_n\}$.
\end{rem}

The following is a key to the proof of Theorem \ref{thm:phase}.

\begin{lem} \label{lem:phase}
  Let $\kappa,\kappa_0,\kappa_1,\dots,\kappa_N$
  be any positive real numbers with $N \ge 1$ such that
  \[
  1-\frac{1}{N}\left(1-\sum_{i=0}^N \kappa_i\right) \le \kappa < 1.
  \]
  Then we have
  \[
  \Sep(X;\kappa_0,\kappa_1,\dots,\kappa_N) \ge \ObsDiam(X;-\kappa)
  \]
  for any mm-space $X$.
\end{lem}

\begin{proof}
  We take any real number $r$ with $0 < r < \ObsDiam(X;-\kappa)$
  and fix it.
  There is a $1$-Lipschitz function $f : X \to \R$ such that
  $\diam(f_*\mu_X;1-\kappa) > r$.  We then have the following:
  \begin{itemize}
  \item[($*$)] If a Borel subset $A \subset \R$ satisfies $\diam(A) \le r$,
    then $f_*\mu_X(A) < 1-\kappa$.
  \end{itemize}
  We define real numbers $a_0,a_1,\dots,a_N$ inductively by
  \begin{align*}
  a_0 &:= \inf\{\;a\in\R \mid f_*\mu_X(\,-\infty,a\,] \ge \kappa_0\;\},\\
  a_i &:= \inf\{\;a \ge a_{i-1}+r \mid f_*\mu_X[\,a_{i-1}+r,a\,] \ge \kappa_i\;\},
  \ i=1,2,\dots,N.
  \end{align*}
  We check the well-definedness of ${a_i}'s$.
  It is clear that $a_0$ is defined as a (finite) real number since $0 < \kappa_0 < 1$.
  Assume that $a_0,a_1,\dots,a_k$ for a number $k \le N-1$
  are defined as real numbers.
  We are going to check that $a_{k+1}$ is defined as a real number.
  For that, it suffices to prove
  \begin{equation} \label{eq:phase}
    f_*\mu_X[\,a_k+r,+\infty\,) > \kappa_{k+1}.
  \end{equation}
  By the definition of $a_i$, we have $f_*\mu_X(\,-\infty,a_0\,) \le \kappa_0$
  and $f_*\mu_X[\,a_{i-1}+r,a_i\,) \le \kappa_i$ for $i=1,2,\dots,k$.
  Also, ($*$) implies $f_*\mu_X[\,a_i,a_i+r\,] < 1-\kappa$ for $i=0,1,\dots,k$.
  We therefore have
  \begin{align*}
  f_*\mu_X(\,-\infty,a_k+r\,)
  &\le \sum_{i=0}^k f_*\mu_X[\,a_i,a_i+r\,]  +
  f_*\mu_X(\,-\infty,a_0\,) \\
  &\quad + \sum_{i=1}^k f_*\mu_X[\,a_{i-1}+r,a_i\,) \\
  &< N(1-\kappa) + \sum_{i=0}^k \kappa_i \\
  &\le 1-\sum_{i=k+1}^N \kappa_i,
  \end{align*}
  which implies \eqref{eq:phase}.

  Setting
  \begin{align*}
  A_0 &:= (\,-\infty,a_0\,],\\
  A_i &:= [\,a_{i-1}+r,a_i\,] \quad\text{for $i=1,2,\dots,N$},
  \end{align*}
  we have $f_*\mu_X(A_i) \ge \kappa_i$ and $d_\R(A_i,A_j) \ge r$
  for $i\neq j$, so that
  \[
  \Sep((\R,f_*\mu_X);\kappa_0,\dots,\kappa_1) \ge r.
  \]
  Since $(\R,f_*\mu_X) \prec X$ and by Proposition \ref{prop:Sep-prec},
  we have
  \[
  \Sep(X;\kappa_0,\dots,\kappa_1) \ge r.
  \]
  By the arbitrariness of $r$, this completes the proof.
\end{proof}

\begin{lem} \label{lem:phase2}
  Let $\cP$ be a pyramid.  Then the following {\rm(1)} and {\rm(2)}
  are equivalent to each other.
  \begin{enumerate}
  \item $\cP$ is weakly dissipated.
  \item $\ObsDiam(\cP;-\kappa) > 0$ for any $\kappa$ with $0 < \kappa < 1$.
  \end{enumerate}
\end{lem}

\begin{proof}
  It is easy to see that
  (1) is equivalent to the following:
  \begin{enumerate}
  \item[(1')] For any $\kappa_0,\dots,\kappa_N > 0$ with $\sum_{i=0}^N \kappa_i < 1$,
  there is an mm-space $X \in \cP$ such that
  $\Sep(X;\kappa_0,\dots,\kappa_N) > 0$.
  \end{enumerate}
  Also, (2) is equivalent to the following:
  \begin{enumerate}
  \item[(2')] For any $\kappa$ with $0 < \kappa < 1$ there is an mm-space
  $X \in \cP$ such that $\ObsDiam(X;-\kappa) > 0$.
  \end{enumerate}

  We prove (1') $\implies$ (2').
  For any given $\kappa$ with $0 < \kappa < 1$,
  we take $\kappa'$ with $0 < \kappa' < \min\{\kappa,1/2\}$.
  Proposition \ref{prop:ObsDiam-Sep} implies
  \[
  \ObsDiam(X;-\kappa) \ge \Sep(X;\kappa',\kappa'),
  \]
  which is positive for some $X \in \cP$ by (1').
  We obtain (2').

  The implication (2') $\implies$ (1')  follows from Lemma \ref{lem:phase}.
  
  This completes the proof of the proposition.
\end{proof}

\begin{lem} \label{lem:phase3}
  Let $\cP_n$, $n=1,2,\dots,$ be pyramids.
  Then the following {\rm(1)} and {\rm(2)} are equivalent to each other.
  \begin{enumerate}
  \item $\{\cP_n\}$ weakly dissipates.
  \item We have
    \[
    \liminf_{n\to\infty} \ObsDiam(\cP_n;-\kappa) > 0
    \]
    for any $\kappa$ with $0 < \kappa < 1$.
  \end{enumerate}
\end{lem}

\begin{proof}
  It follows from Theorem \ref{thm:lim-Sep} and the compactness of $\Pi$ that
  (1) holds if and only if for any weakly convergent subsequence
  of $\{\cP_n\}$, its weak limit is weakly dissipated.
  By Lemma \ref{lem:phase2},
  this is equivalent to that for any weakly convergent subsequence
  of $\{\cP_n\}$, its weak limit, say $\cP$, satisfies
  $\ObsDiam(\cP;-\kappa) > 0$ for any $\kappa$ with $0 < \kappa < 1$,
  which is also equivalent to (2)
  by Theorem \ref{thm:lim-ObsDiam} and the compactness of $\Pi$.
  This completes the proof.
\end{proof}

Using Lemma \ref{lem:phase3}, we present:

\begin{proof}[Proof of Theorem \ref{thm:phase}]
  Let $\{c_n\}$ be a sequence of positive real numbers.
  (1) of Definition \ref{defn:phase} is equivalent to
  \begin{equation} \label{eq:phase1}
    \limsup_{n\to\infty} \ObsDiam(c_n\cP_n;-\kappa) < +\infty
  \end{equation}
  for any $\kappa$ with $0 < \kappa < 1$.
  (2) of Definition \ref{defn:phase} is equivalent to
  the weak dissipation property of $\{c_n\cP_n\}$,
  which is, by Lemma \ref{lem:phase3}, also equivalent to
  \begin{equation} \label{eq:phase2}
    \liminf_{n\to\infty} \ObsDiam(c_n\cP_n;-\kappa) > 0
  \end{equation}
  for any $\kappa$ with $0 < \kappa < 1$.
  Since
  \[
  \ObsDiam(c_n\cP_n;-\kappa) = c_n\ObsDiam(\cP_n;-\kappa),
  \]
  we obtain the theorem.
\end{proof}

\begin{rem}
  We have another way to prove Theorem \ref{thm:phase}
  by generalizing Lemma \ref{lem:phase} for a pyramid.
  However, this is essentially same as above.
\end{rem}

\begin{cor}
  Let $\{\cP_n\}$ be a sequence of pyramids with the phase transition property
  and $\{c_n\}$ a sequence with critical scale order.
  If a sequence $\{t_n\}$ of positive real numbers satisfies $t_n \sim c_n$,
  then $\{t_n\cP_n\}$ neither is a L\'evy family nor $\infty$-dissipates.
\end{cor}

\begin{proof}
  It suffices to prove the corollary for $t_n := c_n$
  because of Propositions \ref{prop:ObsDiam-scale} and \ref{prop:Sep-scale}.

  It follows from \eqref{eq:phase2} that $\{c_n\cP_n\}$ is not a L\'evy family.

  By \eqref{eq:phase1} and Proposition \ref{prop:ObsDiam-Sep-pyramid},
  $\{c_n\cP_n\}$ does not $\infty$-dissipates.
  This completes the proof.
\end{proof}

\begin{proof}[Proof of Corollary \ref{cor:sym-sp}]
  We have
  \[
  \ObsDiam(X_n;-\kappa) \sim 1/\sqrt{n}
  \]
  for $X_n = S^n(1), \RP^n, \CP^n, \HP^n$
  (see \cite{Sy:mmlim}*{Corollaries 5.8 and 5.11} for $S^n(1)$ and $\CP^n$;
  the same proof works for $\RP^n$ and $\HP^n$).
  Theorem \ref{thm:phase}
  proves the phase transition property for $\{S^n(1)\}$, $\{\RP^n\}$, $\{\CP^n\}$,
  and $\{\HP^n\}$.
  
  Since the Ricci curvature of $SO(n)$ is $\sim n$,
  we have
  \[
  \ObsDiam(SO(n);-\kappa) \le O(1/\sqrt{n})
  \]
  (see \cite{Sy:book}*{\S 2.5}).
  By $S^n(1) \prec SO(n)$, we also have
  a lower bound of $\ObsDiam(SO(n);-\kappa)$, so that
  \[
  \ObsDiam(SO(n);-\kappa) \sim 1/\sqrt{n}.
  \]
  This together with Theorem \ref{thm:phase} leads us
  to the phase transition property for $\{SO(n)\}$.
  Since $S^n(1) \prec V_k(\R^n) \prec SO(n)$, we have the phase transition property
  for $\{V_{k_n}(\R^n)\}$.
  The proofs for $\{SU(n)\}$, $\{Sp(n)\}$, $\{V_{k_n}(\C^n)\}$, and $\{V_{k_n}(\H^n)\}$
  are in the same way.
  This completes the proof.
\end{proof}

\begin{defn}
  Let $\alpha > 0$.
  An mm-space is said to be \emph{$\alpha$-atomic}
  if it has an atom with mass $\ge \alpha$.
  A pyramid $\cP$ is said to be \emph{$\alpha$-atomic}
  if any mm-space $X \in \cP$ is $\alpha$-atomic.
  A pyramid (resp.~an mm-space) is \emph{atomic}
  if it is $\alpha$-atomic for some $\alpha > 0$.
  A pyramid (resp.~an mm-space) is \emph{non-atomic} if it is not atomic.
\end{defn}

\begin{prop}
  Let $\{\cP_n\}$ be a sequence of pyramids with the phase transition property
  and $\{c_n\}$ a sequence  of critical scale order for $\{\cP_n\}$.
  Then, the limit, say $\cP$, of any weakly convergent subsequence of $\{c_n\cP_n\}$
  satisfies
  \begin{equation} \label{eq:crit-lim}
    0 < \ObsDiam(\cP;-\kappa) < +\infty
  \end{equation}
  for any $\kappa$ with $0 < \kappa < 1$.
  In particular, $\cP$ is non-atomic.
\end{prop}

\begin{proof}
  It follows from the phase transition property and Theorem \ref{thm:phase} that
  \begin{align*}
    \liminf_{n\to\infty} \ObsDiam(c_n\cP_n;-\kappa) &> 0, \\
    \limsup_{n\to\infty} \ObsDiam(c_n\cP_n;-\kappa) &< +\infty
  \end{align*}
  for any $\kappa$ with $0 < \kappa < 1$.
  These inequalities together with Theorem \ref{thm:lim-ObsDiam}
  implies \eqref{eq:crit-lim}.
 
  We prove that $\cP$ is non-atomic.
  We see that if an mm-space $X$ is $\alpha$-atomic for a real number $\alpha > 0$,
  then
  \[
  \ObsDiam(X;-\kappa_\alpha) \le \diam(X;1-\kappa_\alpha) = 0
  \]
  for $\kappa_\alpha := 1-\alpha$.
  Therefore, if $\cP$ is $\alpha$-atomic for some $\alpha > 0$,
  then we have $\ObsDiam(\cP;-\kappa_\alpha) = 0$,
  which is a contradiction to \eqref{eq:crit-lim}.
  This completes the proof.
\end{proof}

\end{document}